\documentclass[11pt]{amsart}
\usepackage{graphicx}
\usepackage[dvips]{epsfig}
\usepackage{pinlabel}
\usepackage{amsmath}
\usepackage{amsfonts}
\usepackage{latexsym}
\usepackage{amssymb}
\usepackage[usenames]{color}
\usepackage{amsthm}
\usepackage[all]{xypic}
\usepackage{enumitem}
\usepackage{mathrsfs}
\usepackage{mathtools} 
\usepackage[breaklinks=true]{hyperref}

\input xy
\xyoption{all}
\def\E{\ifmmode{\mathbb E}\else{$\mathbb E$}\fi} 
\def\N{\ifmmode{\mathbb N}\else{$\mathbb N$}\fi} 
\def\R{\ifmmode{\mathbb R}\else{$\mathbb R$}\fi} 
\def\Q{\ifmmode{\mathbb Q}\else{$\mathbb Q$}\fi} 
\def\C{\ifmmode{\mathbb C}\else{$\mathbb C$}\fi} 
\def\H{\ifmmode{\mathbb H}\else{$\mathbb H$}\fi} 
\def\Z{\ifmmode{\mathbb Z}\else{$\mathbb Z$}\fi} 
\def\P{\ifmmode{\mathbb P}\else{$\mathbb P$}\fi} 
\def\T{\ifmmode{\mathbb T}\else{$\mathbb T$}\fi} 
\def\SS{\ifmmode{\mathbb S}\else{$\mathbb S$}\fi} 
\def\DD{\ifmmode{\mathbb D}\else{$\mathbb D$}\fi} 

\newcommand{\ben}{\begin{enumerate}}
\newcommand{\een}{\end{enumerate}}
\newcommand{\be}{\begin{equation}}
\newcommand{\ee}{\end{equation}}
\newcommand{\bea}{\begin{eqnarray}}
\newcommand{\eea}{\end{eqnarray}}
\newcommand{\bc}{\begin{center}}
\newcommand{\ec}{\end{center}}

\newtheorem{thm}{Theorem}[section]
\newtheorem{cor}[thm]{Corollary}
\newtheorem{lem}[thm]{Lemma}
\newtheorem{prop}[thm]{Proposition}

\theoremstyle{definition}
\newtheorem{defn}{Definition}[section]

\theoremstyle{remark}
\newtheorem{rem}{\rm\bfseries{Remark}}[section]

\newtheorem{exm}[rem]{\rm\bfseries{Example}}

\newcommand{\OP}{\operatorname}
\newcommand{\CP}{\ensuremath{\C}P}
\DeclareMathAlphabet{\mathdj}{U}{msb}{m}{n}

\newcommand{\id}{\operatorname{Id}}

\newcommand{\pt}{{\operatorname{pt}}}

\begin{document}


\title[Bulky Hamiltonian isotopies of tori]{Bulky Hamiltonian isotopies of Lagrangian tori with applications}

\author[DIMITROGLOU RIZELL]{Georgios Dimitroglou Rizell}

\thanks{The author is supported by the grant KAW 2016.0198 from the Knut and Alice Wallenberg Foundation.}

\address{Department of Mathematics, Uppsala University, Box 480, SE-751 06, Uppsala, Sweden.}
\email{georgios.dimitroglou@math.uu.se}

\begin{abstract}
We exhibit monotone Lagrangian tori inside the standard symplectic four-dimensional unit ball that become Hamiltonian isotopic to the Clifford torus, i.e.~the standard product torus, only when considered inside a strictly larger ball (they are not even symplectomorphic to a standard torus inside the unit ball). These tori are then used to construct new examples of symplectic embeddings of toric domains into the unit ball which are symplectically knotted in the sense of J.~Gutt and M.~Usher. We also give a characterisation of the Clifford torus inside the ball as well as the projective plane in terms of quantitative considerations; more specifically, we show that a torus is Hamiltonian isotopic to the Clifford torus whenever one can find a symplectic embedding of a sufficiently large ball in its complement.
\end{abstract}


\maketitle

\section{Introduction and results}
The focus in this paper is on two-dimensional Lagrangian tori inside the standard four-dimensional Liouville manifold $(\C^2,\omega_0=d\lambda_0)$ as well as the projective plane $(\CP^2,\omega_{\OP{FS}})$. We begin by recalling the basic definitions.

The Liouville manifold
$$(\C^2,\omega_0=dx_1\wedge dy_1 + dx_2\wedge dy_2)$$
is equipped with the standard Liouville form
$$\lambda_0=\frac{1}{2}\sum_{i=1}^2(x_idy_i-y_idx_i), \:\: d\lambda_0=\omega_0.$$
 Denote by $\zeta_0$ the corresponding Liouville vector field, which generates the flow
\begin{gather*}
\phi^t_{\lambda_0} \colon (\C^2,d\lambda_0) \to (\C^2,e^{-t}d\lambda_0),\\
\phi^t_{\lambda_0}(\mathbf{z})=e^{t/2}\mathbf{z}.
\end{gather*}
To set up notation, we will denote by $B^{2n}_x(r)$ and $D^{2n}_x(r)$ the open and closed balls inside $\C^n$, respectively, which are centred at the point $x \in \C^n$ and of radius $r>0$. Further, denote by $S^{2n-1}(r) \coloneqq \partial D^{2n}_0(r) \subset \C^n$ the sphere of radius $r>0.$ We also write $B^{2n} \coloneqq B^{2n}_0(1)$, $D^{2n} \coloneqq D^{2n}_0(1)$, and $S^{2n-1} \coloneqq S^{2n-1}(1)$.

The projective plane $(\CP^2,\omega_{\OP{FS}})$ equipped with the standard Fubini--Study form can be obtained by collapsing the boundary $S^3$ of $(D^4,\omega_0)$ to a line $\ell_\infty \cong \CP^1$ via symplectic reduction; in particular, the symplectic area of the line class $[\ell] \in H_2(\CP^2)$ is then equal to $\int_\ell \omega_{\OP{FS}}=\pi$.

A {\bf Lagrangian} submanifold is a half-dimensional manifold on which the symplectic form vanishes. In the case of $\C^2$ the last condition is equivalent to the requirement that $\lambda_0$ pulls back to a closed form. A {\bf Lagrangian isotopy} is a smooth isotopy through Lagrangian embeddings; recall the standard fact that such an isotopy inside $\C^2$ can be generated by a global {\bf Hamiltonian isotopy} of the ambient symplectic manifold if and only if the pullbacks of $\lambda_0$ are constant in cohomology; see e.g.~\cite{Weinstein:Berry} by A.~Weinstein. In general, we will call a smooth isotopy of a subset of a symplectic manifold {\bf Hamiltonian} if it can be realised by a Hamiltonian isotopy of the ambient symplectic manifold.

The {\bf symplectic action class} of a Lagrangian inside $\C^2$ is the cohomology class $\sigma_L := [\lambda_0|_{TL}] \in H^1(L;\R)$ pulled back to it, which by Stokes' theorem is the value of the symplectic area of any two-chain inside $\C^2$ with boundary in that class. A torus is {\bf monotone} if the symplectic area of a two-dimensional chain with boundary on it is proportional to the so-called Maslov class of the same chain; see V.~Arnold~\cite{Arnold:Maslov} for the definition of the latter characteristic class. In particular, it follows that the Lagrangian {\bf product tori} $S^1(a) \times S^1(b) \subset \C^2$ are monotone if and only if $a=b.$ Note that the symplectic action class of he standard monotone product tori $S^1(r) \times S^1(r) \subset (\C^2,\omega_0)$ take values that are integer multiples of $\pi r^2>0.$ These tori are usually called {\bf Clifford tori}. 

R.~Vianna~\cite{Vianna:Second} has shown that the classes of monotone Lagrangian tori inside $(\CP^2,\omega_{\OP{FS}})$ exhibit a very rich and interesting structure. In particular, there exists \emph{infinitely} many different Hamiltonian isotopy classes of such tori. The result~\cite[Theorem C]{Dimitroglou:Isotopy} by the author together with E.~Goodman and A.~Ivrii implies that all of Vianna's tori can be placed inside the open unit ball
$$(B^4,\omega_0) = (\CP^2 \setminus \ell_\infty,\omega_{\OP{FS}})$$
after a Hamiltonian isotopy. It thus follows a fortiori that his construction also gives rise to infinitely many different Hamiltonian isotopy classes of monotone Lagrangian tori inside $B^4.$ In contrast to this, the only known Hamiltonian isotopy classes of Lagrangian tori inside the complete Liouville manifold $(\C^2,\omega_0) \supset B^4$ are the product tori, together with linear rescalings of the ``exotic'' monotone torus~\cite{Chekanov:LagrangianTori} constructed by Y.~Chekanov; the latter goes under the name of the {\bf Chekanov torus}, and we refer to the work~\cite{Gadbled:OnExotic} by A.~Gadbled for the presentation that we will use here.

We expect that Vianna's tori all become Hamiltonian isotopic to standard tori inside a ball which is strictly larger than the unit ball. (This can be confirmed by hand for e.g.~certain particular Hamiltonian isotopies that take Vianna's first exotic torus constructed in~\cite{Vianna:First} into the unit ball.)

\begin{rem}
Even though all Lagrangian tori are Lagrangian isotopic inside the ball by~\cite{Dimitroglou:Isotopy}, there is still no classification of Lagrangian tori inside the plane up to \emph{Hamiltonian} isotopy. Under additional assumptions concerning a certain linking behaviour with a conic; the author established a Hamiltonian classification in~\cite{Dimitroglou:Whitney}.
\end{rem}

Our first result is a criterion for when a monotone Lagrangian torus is Hamiltonian isotopic to a Clifford torus in terms of the existence of a large symplectic ball in its complement.
\begin{thm}
  \label{thm:1}
  \begin{enumerate}
    \item 
Let $L \subset (B^4,\omega_0)$ be a Lagrangian torus inside the unit ball whose whose symplectic action class takes the values $\Z \pi r^2$ on $H_1(L)$ for some fixed $r\ge 1/\sqrt{3}.$ There exists a Hamiltonian isotopy inside the ball which takes $L$ to the standard monotone product torus $S^1(r) \times S^1(r)$ if and only if it is disjoint from the interior of some symplectic embedding of $(D^4(\sqrt{2/3}),\omega_0)$ into $(B^4,\omega_0)$.
\item A monotone Lagrangian torus $L \subset (\CP^2,\omega_{\OP{FS}})$ which is disjoint from the interior of some symplectic embedding of $(D^4(\sqrt{2/3}),\omega_0)$ is Hamiltonian isotopic to the standard Clifford torus
  $$S^1(\sqrt{1/3})\times S^1(\sqrt{1/3}) \subset B^4 = (\CP^2 \setminus \ell_\infty,\omega_{\OP{FS}})$$
 contained in the affine chart.
\end{enumerate}
\end{thm}

We then show that Part (1) of the above theorem is sharp in the following sense: Even under the stronger assumption that $L \subset B^4 \setminus B^4(\sqrt{2/3}-\epsilon)$ is a Lagrangian torus that is Hamiltonian isotopic to $S^1(r) \times S^1(r)$ inside the full plane $(\C^2,\omega_0),$ there are cases when any such Hamiltonian isotopy must intersect $S^3=\partial D^4$ at some moment in time. (In other words, the Hamiltonian isotopy cannot be confined to the unit ball that contains the original Lagrangian.) More precisely, we establish that
\begin{thm}
\label{thm:2}
There exists a Lagrangian torus $L \subset (B^4,\omega_0)$ which is Hamiltonian isotopic inside $(\C^2,\omega_0)$ to the standard product torus
$$S^1(1/\sqrt{3}) \times S^1(1/\sqrt{3}),$$
but where every such Hamiltonian isotopy necessarily satisfies $\phi^{t_0}_{H_t}(L) \cap S^3 \neq \emptyset$ for at least one value $t_0 \in[0,1].$ In addition, we may assume that one of the following holds:
\begin{itemize}
\item $L \subset B^4 \setminus B^4(\sqrt{2/3}-\epsilon),$ whenever $\epsilon>0$ is sufficiently small, or
\item $L \subset B^4(\sqrt{1-r}) \subset B^4,$ whenever $r \in (0,1/6).$
\end{itemize}
\end{thm}
The torus $L$ is constructed in Section~\ref{sec:ProofThm2} by explicit means involving a probe, which is a tool that was invented in \cite{McDuff:Probes} by D.~McDuff; see Figures~\ref{fig:probe1} and \ref{fig:probe2} for its depiction. These examples are thus of a more elementary kind than the tori constructed by Vianna (which are rather cumbersome to describe explicitly inside the ball). In fact, the example that we consider can be identified with the monotone Chekanov torus inside $\CP^2,$ but where the embedding of $B^4 \hookrightarrow \CP^2$ that contains the torus is obtained by removing a line in the complement of the torus which is different from the ``standard line at infinity''. In order to distinguish $L$ from a product torus inside the \emph{unit ball} it suffices to compactify the ball to $\CP^2=\overline{B^4},$ and then to use the classical result by Y.~Chekanov and F.~Schlenk~\cite{Chekanov:TwistTori} that the monotone Chekanov torus is not Hamiltonian isotopic to a product torus inside $\CP^2.$ See Section \ref{sec:holomorphic} for a discussion about how the holomorphic curves distinguish $L$ from the product torus in this case.

Additionally, in conjunction with Theorem \ref{thm:1}, we can conclude that $L$ is exotic also in the following sense which (at least a priori) is stronger:
\begin{cor}
\label{cor:2}
The Lagrangian torus $L$ in Theorem \ref{thm:2} not in the image of $S^1(1/\sqrt{3}) \times S^1(1/\sqrt{3})$ under any symplectomorphism $(B^4,\omega_0) \xrightarrow{\cong} (B^4,\omega_0).$
\end{cor}

\subsection{Gromov width of the complement of Lagrangians} 
The Gromov width is a well-studied symplectic capacity that was introduced by M.~Gromov in~\cite{Gromov:Pseudo}, which for a symplectic manifold $(X^4,\omega)$ is equal to the supremum
$$ \sup \left\{ \pi\cdot r^2; \:\: \exists \varphi \colon (B^4(r),\omega_0) \hookrightarrow (X,\omega)\right\} $$
taken over all symplectic embeddings of open balls.

There are previous computations of the Gromov width of the complement of certain Lagrangian submanifolds; we refer to~\cite{Biran:LagrangianBarriers} by P.~Biran as well as~\cite[Section 6.2]{QuantumHomology} by P.~Biran and O.~Cornea, who showed that the Gromov width of the complement of the monotone Clifford torus inside $\CP^2$ is equal to $\pi 2/3$. In \cite[Theorem 1.6]{Cieliebak:PuncturedHolomorphic} K.~Cieliebak and K.~Mohnke give a strong bound for the Gromov width of the complement of arbitrary Lagrangian tori in $\CP^n$; their bound in particular implies that the complement of an arbitrary monotone torus $L \subset \CP^2$ has Gromov bounded from above by $\pi 2/3$.

In the more recent work \cite{Lee:Asymptotic} by W. Lee, Y.-G. Oh, and R. Vianna it is shown that all of Vianna's tori have Gromov width at least equal to $\pi/3$ with our convention for $\omega_{\OP{FS}}$. (Warning: in the latter paper the Fubini--Study form on $\CP^2$ is rescaled by a multiple of two relative our convention.)
\begin{rem}Part (2) of Theorem \ref{thm:1} gives a partial answer to \cite[Conjecture 4.2]{Lee:Asymptotic}, since it shows that the Clifford torus is the unique monotone torus inside $(\CP^2,\omega_{\OP{FS}})$ that admits an open symplectic ball $(B^4(\sqrt{2/3}),\omega_0)$ in its complement, albeit under the additional assumption that the embedding has a symplectic extension to the closed ball.

The latter condition could be removed, if it is true that any monotone Lagrangian torus contained inside the closure of such an open symplectic ball is Hamiltonian isotopic to a product torus (which is the case whenever there is a symplectic extension of the embedding to the closed ball $D^4(\sqrt{2/3})$).
\end{rem}

On the negative side, we note that the Gromov width of the complement fails to distinguish at least some of the monotone tori in $\CP^2$:
\begin{prop}
The monotone Clifford torus, Chekanov torus, as well as Vianna's first torus from \cite{Vianna:First} inside $\CP^2$ (i.e.~$T(1,1,1)$, $T(1,1,4)$ and $T(1,4,25)$ using Vianna's notation) all have complements with Gromov width $\pi 2/3$.
\end{prop}
\begin{proof}
  For the Clifford torus we can take the obvious symplectic ball centered at the origin. For the Chekanov torus we can find symplectic balls of radius $s$ for any $s < \sqrt{2/3}$; see Figure \ref{fig:Chekanov}. For the torus $T(1,4,25)$ we leave it to the reader to check that the construction of $T(1,4,25)$ from \cite{Vianna:First}, i.e.~the mutation of $T(1,1,4)$ along a suitable embedded Lagrangian disc, can be performed in the complement of these balls. The key point is that one of the two Lagrangian discs considered by Vianna itself lives in the complement of these balls.
\end{proof}

\begin{figure}[htp]
\centering
\vspace{3mm}
\labellist
\pinlabel $1$ at 100 -8
\pinlabel $u_1/\pi$ at 124 3
\pinlabel $u_2/\pi$ at 4 118
\pinlabel $r^2+\epsilon$ at 57 -6
\pinlabel $1/3$ at -11 35
\pinlabel $1/2$ at -11 51
\pinlabel $1$ at -5 100
\pinlabel $1/3$ at 28 -8
\pinlabel $B$ at 60 25
\pinlabel $\color{red}L_{\OP{Cl}}$ at 30 45
\pinlabel $\ell_\infty$ at 30 87
\pinlabel $\color{blue}L_{\OP{Ch}}$ at 34 20
\endlabellist
\includegraphics{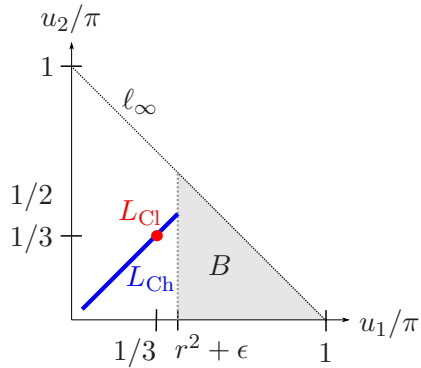}
\vspace{3mm}
\caption{For any $\epsilon>0$, the version of the Chekanov torus $L_{\OP{Ch}}$ inside $B^4$ whose symplectic action assumes the values $\Z\pi r^2$ with $r \in [1/\sqrt{3},1)$ can be placed over the diagonal inside the region $\{u_1/\pi < r^2+\epsilon\}$. (When $r=1/3$ this is the monotone Chekanov torus inside $\CP^2$.) For any neighbourhood of $B=\{u_2/\pi \ge r^2+\epsilon\}$ there exists a symplectic embedding of $B^4(\sqrt{1-r^2-\epsilon})$ which moreover can be taken to be disjoint from $\ell_\infty$; see L.~Traynor's work \cite{Traynor:Packing}.}
\label{fig:Chekanov}
\end{figure}

It is interesting that for tori in $B^4$, the Gromov can in fact be used to distinguish certain ``large'' tori which are monotone \emph{inside the ball}.
\begin{prop}
  \begin{enumerate}
  \item The Gromov width of $B^4 \setminus S^1(r) \times S^1(r)$ is equal to $\pi2r^2$ whenever $r \in [\sqrt{1/3},1)$.
  \item If a Lagrangian torus $L \subset B^4$ has symplectic action which assumes the values $\Z \pi r^2$ for some fixed $r \in [1/\sqrt{3},1)$ and satisfies the property that the Gromov width of $B^4 \setminus L$ is strictly greater than $\pi 2/3$, then it is Hamiltonian isotopic to a product torus.
  \end{enumerate}
\end{prop}
\begin{proof}
  (1): Considering the round balls centered at the origin we conclude that the Gromov width is bounded from below by $\pi 2r^2$. On the other hand, it is also bounded from above by the same number by \cite[Theorem 1.6]{Cieliebak:PuncturedHolomorphic}. (Alternatively one could argue along the lines of \cite[Corollary 6.5]{QuantumHomology}.)
  
  (2): This is a direct consequence of Part (1) of Theorem \ref{thm:1}.
\end{proof}
\begin{exm} The Chekanov torus inside $(B^4,\omega_0)$ whose symplectic action class assumes the values $\Z \pi r^2$ admits balls in its complement of radius $\sqrt{s}$ for any $s<2/3-r$; see Figure \ref{fig:Chekanov}. Its complement thus has Gromov width equal to some value $w \in [\pi (2/3-r),\pi 2/3]$. Even though we cannot give a more precise computation of this Gromov width, it is still strictly less than the Gromov width of the complement of the product torus of the same symplectic action.
\end{exm}

\subsection{Application to knotted symplectic embeddings}

A typical symplectic embedding problem concerns the question whether there exists an embedding
$$(Y^{2n},d\lambda_Y) \hookrightarrow (X^{2n},Cd\lambda_X)$$
of a symplectic manifold into e.g.~an open Liouville domain $(X,C\lambda_X)$ for some $C> 0.$ Here we assume that $X$ is the interior of a compact Liouville domain with smooth boundary,  while $Y$ is compact subset of a symplectic domain with a sufficiently well-behaved boundary. Typically one is interested in the case when an obvious, or even canonical, such embedding exists for all $C \gg 0$ sufficiently large. The natural question is then: how small can $C>0$ be taken for a symplectic embedding to exist?

The first nontrivial result about symplectic embeddings was Gromov's famous non-squeezing result~\cite{Gromov:Pseudo}, which showed that there are interesting symplectic obstructions beyond the obvious volume obstruction. Since then symplectic embedding problems have received a great amount of attention. Particularly in dimension four, the situation is rather well understood for some special cases of $Y$ and $X$. One notable such instance is the seminal work~\cite{McDuff:HoferConjecture} by D.~McDuff, which answers the question when an ellipsoid can be embedded into a ball.

Many of the natural examples of domains of the form $Y \subset (\C^n,\omega_0)$ that have been studied in the literature have the feature that $\partial Y$ is foliated by (possibly degenerate) Lagrangian standard product tori. Most attention has been given to domains for which the standard Liouville vector field $\zeta_0$ moreover is transverse to $\partial Y.$ Such domains include closed balls $D^n(r),$ closed ellipsoids $E(a,b) := \{ \pi\|z_1\|^2/a+\pi\|z_2\|^2/b \le 1\},$ as well as polydiscs $D^2(a) \times D^2(b)$ (the latter has a smooth boundary with corner equal to a Lagrangian product torus). Domains of this type are typically depicted by their image under the standard momentum map
\begin{gather*}
\mu \colon \C^2 \to \R^2,\\
(z_1,z_2) \mapsto (u_1,u_2)=\pi(\|z_1\|^2,\|z_2\|^2).
\end{gather*}
In this manner we obtain a direct connection between symplectic embedding problems and embedding problems for families of Lagrangian tori. This direction was taken in the work~\cite{Hind:SymplecticEmbeddings} by R.~Hind and S.~Lisi, \cite{Cieliebak:PuncturedHolomorphic} by K.~Cieliebak and K.~Mohnke, \cite{Gutt:SymplecticCapacities} by J.~Gutt and M.~Hutchings, and \cite{Hind:Squeezing} by R.~Hind and E.~Opshtein.

In the case when there exists a symplectic embedding $(Y^{2n},d\lambda_Y) \hookrightarrow (X^{2n},d\lambda_X)$ one can further ask the question whether two different such embeddings have images that can be made to coincide after a symplectomorphism of the ambient space $(X,d\lambda_X).$ It was shown by J.~Gutt and M.~Usher \cite{Gutt:SymplecticallyKnotted} that this is not necessarily the case, even if such a symplectomorphism exists for the completion of $(X,d\lambda_X)$ to a Liouville manifold $(\overline{X},d\lambda_X)$. The same authors calls an embedding {\bf symplectically knotted} (relative some other embedding) if there exits an ambient symplectomorphism inside the completed Liouville domain that takes the image of one embedding to the other, but when no such symplectomorphism exists of the original Liouville domain.

We now show that, in view of Corollary \ref{cor:2}, the embedding of a domain can be shown to be symplectically knotted by considering Lagrangian tori contained inside its boundary.
\begin{thm}
\label{thm:knotted}
Let
$$Y \subset \{ \|z_1\|^2+2\|z_2\|^2 \le 1\} \subset (D^4,\omega_0)$$
be any closed symplectic domain that satisfies
$$ Y \cap \{ \|z_1\|^2+2\|z_2\|^2 = 1\} = S^1(1/\sqrt{3}) \times S^1(1/\sqrt{3}).$$
There exists a symplectic embedding
$$\phi \colon (Y,\omega_0) \hookrightarrow (B^4,\omega_0)$$
for which the monotone Lagrangian torus $S^1(1/\sqrt{3}) \times S^1(1/\sqrt{3}) \in \partial Y$ is mapped to a torus $L$ as in Theorem \ref{thm:2} and such that, for any $\epsilon>0,$ one can find a symplectomorphism $\Phi \colon (B^4(1+\epsilon),\omega_0) \to (B^4(1+\epsilon),\omega_0)$ that satisfies $\Phi(\phi(Y))=Y$.
\end{thm}
The construction of the symplectic embedding is a slight extension of the construction of the Lagrangian $L$ in the proof of Theorem \ref{thm:2}, and it is performed in Section \ref{proof:sympknotted}. In view of Corollary \ref{cor:2} combined with \cite[Theorem 1.2]{Dimitroglou:Extremal} we can now conclude that:
\begin{cor}
Assume that
$$Y \subset D^4(\sqrt{2/3})=\{ \|z_1\|^2+\|z_2\|^2 \le 2/3\}$$
is satisfied in addition to the above. Then the embedding $\phi(Y) \subset (B^4,\omega_0)$ is symplectically knotted relative the canonical inclusion $Y \subset (B^4,\omega_0).$
\end{cor}
\begin{proof}
  Any symplectomorphism $\Phi$ that takes $\phi(Y)$ to $Y$ takes the Lagrangian torus $\phi(S^1(1/\sqrt{3}) \times S^1(1/\sqrt{3}))=L$ to a Lagrangian torus that is contained inside $D^4(\sqrt{2/3})$ and which is extremal in the sense of \cite{Cieliebak:PuncturedHolomorphic} relative the latter ball. By \cite[Theorem 1.2]{Dimitroglou:Extremal} it then follows that $L$ must be contained entirely inside the boundary
  $$\partial D^4(\sqrt{2/3})=S^3(\sqrt{2/3})$$
  and, by the assumptions on the domain, it must thus be equal to the standard product torus
  $$ Y \cap S^3(\sqrt{2/3})=S^1(1/\sqrt{3}) \times S^1(1/\sqrt{3}).$$
Assuming the existence of such a symplectomorphism $\Phi \colon B^4 \xrightarrow{\cong} B^4$ we can thus deduce the existence of a symplectic embedding of a closed ball $\Phi^{-1}D^4(\sqrt{2/3}) \subset B^4$ which intersects the exotic torus $L$ only along its boundary. Using Part (1) of Theorem \ref{thm:1} we obtain a Hamiltonian isotopy from $L$ to a standard product torus inside $B^4$. This yields the sought contradiction with Theorem \ref{thm:2}
\end{proof}

\begin{exm}
The above method in particularly yields a symplectically knotted embedding of the polydisc $Y=D^2(1/\sqrt{3}) \times D^2(1/\sqrt{3}) \subset (B^4,\omega_0)$ into the unit ball. This domain was not considered in \cite{Gutt:SymplecticallyKnotted}, and seems to be of a rather different nature than the examples therein. It is unclear to the author if this embedding remains symplectically knotted also inside $B^2(1) \times B^2(1)$; if this is the case, then it would answer Question 1.9 in the aforementioned paper.
\end{exm}

\subsection{Proposed notion: Bulky Hamiltonian isotopy}
In view of the previous theorem, we find the following proposed definition to be natural. Consider two subsets $A_0, A_1 \subset (X^{2n},d\lambda)$ of a compact Liouville cobordism with smooth boundary $\partial X=\partial_+ X \sqcup \partial X_-$ (the latter being the decomposition into its convex and concave components), and denote by $(\overline{X}^{2n},d\lambda)$ the completion of $(X,d\lambda)$ to a Liouville cobordism with noncompact cylindrical ends. In particular,
$$(\overline{X} \setminus (X\setminus \partial X),d\lambda) \cong ([0,+\infty) \times \partial_+ X \sqcup (-\infty,0] \times \partial_- X,d(e^s\lambda|_{T\partial X}))$$
where the latter exact symplectic manifold is the symplectisation of the boundary of $(X,d\lambda).$
\begin{defn}
A Hamiltonian isotopy from $A_0$ to $A_1$ inside the completion of $(X,d\lambda),$ i.e.~a Hamiltonian isotopy
$$\phi_{H_t}^t \colon (\overline{X},d\lambda)\to(\overline{X},d\lambda)$$
which satisfies
$$ \phi_{H_t}^0=\id,\:\:\text{and}\:\:\phi_{H_t}^1(A_0)=A_1,$$
is said to be a \emph{bulky  Hamiltonian isotopy from $A_0$ to $A_1$ relative $X$} if there exists no smooth one-parameter family $\phi^{t,s}$ of Hamiltonian isotopies of the same kind that satisfies $\phi^{t,0}=\phi_{H_t}^t$, while
$$\phi^{t,1}(A_0) \subset X$$
holds for all $t \in [0,1].$
\end{defn}
In other words, we can rephrase the Theorem \ref{thm:2} as the statement that ``the Hamiltonian isotopies in $\C^2$ that take $L$ to a standard torus are all bulky relative the unit ball.''

We end with an additional example, again of rather elementary nature.
\begin{exm}
  Again consider the monotone product torus
  $$L_{\OP{Cl}}=S^1(1/\sqrt{3}) \times S^1(1/\sqrt{3}).$$
  There exists a Hamiltonian isotopy which takes $L$ to itself, through tori contained entirely inside $S^3(\sqrt{2/3})$, with the result that the two factors of the torus become interchanged; e.g.~take a suitable path of matrices in $U(2)$ which start with the identity and end with $(z_1,z_2) \mapsto (z_2,z_1)$. However, there exists no such Hamiltonian isotopy which is contained entirely inside $B^2(\sqrt{2/3}) \times B^2(\sqrt{2/3})$. This can be seen by the invariance of the two terms of the superpotential of the monotone torus $$L_{\OP{Cl}} \subset (\CP^1 \times \CP^1,(2/3)\omega_{\OP{FS}} \oplus (2/3)\omega_{\OP{FS}})$$
  corresponding to the Maslov-two discs that pass through the divisor
  $$\{\infty\} \times \CP^1 \: \cup \: \CP^1 \times \{\infty\}.$$
  More precisely, we need to consider a version of the superpotential which also keeps track of the relative homology classes if the discs. We again refer to ~\cite{Chekanov:TwistTori} for the computation of the superpotential, and note that the relative homology classes of these discs are invariant for Hamiltonian isotopies of the torus that stay away from the divisor. Rephrased using our newly defined notion: ``the Hamiltonian isotopies inside $\C^2$ that interchange the two factors of $L_{\OP{Cl}}$ are bulky with respect to any Liouville domain contained inside $B^2(\sqrt{2/3}) \times B^2(\sqrt{2/3})$.''
\end{exm}
  
\subsection{Holomorphic curve invariants}
\label{sec:holomorphic}
As previously mentioned, we rely on the work \cite{Chekanov:TwistTori} in order to show the nonexistence of a Hamiltonian isotopy inside $B^4$ which takes the product torus to $L$ in Theorem \ref{thm:2}. It is worthwhile to elaborate a bit on the precise mechanism which distinguishes between these tori. The tool used in~\cite{Chekanov:TwistTori} for distinguishing different Hamiltonian isotopy classes of Lagrangian tori is the theory of pseudoholomorphic curves, more precisely by comparing the superpotentials of the different tori. The superpotential is an invariant which counts the number of families of pseudoholomorphic Maslov-two discs with boundary on the torus; see e.g.~the work~\cite{Auroux:SpecialLagrangian} by D.~Auroux.

The lesson that we learn from the examples in Theorem \ref{thm:2} is that, even though we are interested in obstructing Hamiltonian isotopy inside the ball, it is completely crucial that we consider the superpotential which counts holomorphic discs in all of $\CP^2=\overline{B^4}$. If instead the superpotential inside the ball was to be considered, it would give the same answer for $L$ and the product torus; the reason is that, for a monotone Lagrangian torus inside the ball, the superpotential is invariant under Hamiltonian isotopy as well as linear rescalings.

In conclusion, for the torus $L$ considered here, it is the terms in the superpotential that count the discs in $\CP^2$ passing through the line at infinity $\ell_\infty=\CP^2 \setminus B^4$ that distinguish it from the product torus. In general this is not a well-defined count if the torus is merely assumed to be monotone inside the ball (being monotone inside $\CP^2$ is a stronger condition). However, in our case the count of Maslov-two discs that pass through the line at infinity is well-defined for the following reason. For an almost complex structure on $\CP^2=\overline{B^4}$ which makes the line at infinity holomorphic, the class of pseudoholomorphic discs of Maslov index two that pass through the line at infinity are a priori of \emph{minimal} symplectic area, given the symplectic action properties of the tori $L$ under consideration. Hence, the count of these discs is invariant under deformations of the almost complex structure that keeps the line at infinity holomorphic. However, for a ball which is larger than the unit ball, the corresponding discs are no longer of minimal symplectic area. Since we then cannot exclude bubbling from occurring while varying the almost complex structure, we no longer have any reasons to expect that the number of such disc is an invariant that can be used to obstruct the existence of a Hamiltonian isotopy. (Indeed, Theorem \ref{thm:1} can provide a Hamiltonian isotopy in this case.)

\section{The proof of Part (1) of Theorem \ref{thm:1}}

Denote by
\begin{gather*}
\varphi \colon (D^4(\sqrt{2/3}),\omega_0) \hookrightarrow (B^4,\omega_0),\\
L \subset B^4 \setminus \varphi(B^4(\sqrt{2/3})),
\end{gather*}
the symplectic embedding whose existence is assumed.

After the application of the positive Liouville flow $\phi^{4\epsilon}_{\lambda_0}$ to both $L$ and $\varphi(B^4(\sqrt{2/3}))$ for some small $\epsilon>0$ (recall that $\phi^{4\epsilon}_{\lambda_0}$ is the conformal symplectomorphism given by scalar multiplication with $e^{2\epsilon}$) we may in addition assume that the new Lagrangian torus has the symplectic actions $\Z\pi e^{4\epsilon}r^2$ while, of course, it now is disjoint from the rescaled image $e^{2\epsilon}\varphi(B^4(\sqrt{2/3}))$ of a symplectic ball of slightly larger radius $e^{2\epsilon}\sqrt{2/3}$. (Here we have made use of the assumption in Theorem \ref{thm:1} that the closure of the image of $\varphi$ is contained inside the \emph{open} unit ball.) If we manage to construct the sought Hamiltonian isotopy after the above rescaling, the general case will then also follow immediately. Indeed, it suffices to rescale the produced Hamiltonian isotopy by the negative-time Liouville flow $\phi^{-4\epsilon}_{\lambda_0}.$

In view of the above, we will now restrict attention to the case when $r>1/\sqrt{3}$ and $\varphi \colon (B^4(e^{2\epsilon}\sqrt{2/3}),\omega_0) \hookrightarrow (B^4 \setminus L,\omega_0)$ satisfies
$$\overline{\varphi(B^4(e^{\epsilon}\sqrt{2/3}),\omega_0)} \subset B^4 \setminus L,$$
i.e.~we can find a closed ball of radius $e^\epsilon\sqrt{2/3}$ in the complement of $L$.

\subsection{A neck-stretching sequence}
\label{sec:neckstretch}
Recall that symplectic reduction applied to the boundary $\partial D^4=S^3 \to \CP^1$ produces a compactification $\overline{B^4} = \CP^2$ where the latter is equipped with the Fubini--Study symplectic form $\omega_{\OP{FS}}$ for which a line has symplectic are equal to $\int_\ell \omega_{\OP{FS}}=\pi$. In particular, using $\ell_\infty$ to denote the line at infinity, we have $(B^4,\omega_0)=(\CP^2 \setminus \ell_\infty,\omega_{\OP{FS}})$.

The main technical ingredient that we will need is neck-stretching around a hypersurface of contact type that can be identified with a small unit normal bundle around $L.$ Neck-stretching first appeared in work \cite{Eliashberg:SFT} by Y.~Eliashberg, A.~Givental, and H.~Hofer, and was later made precise in the SFT compactness theorem~\cite{Bourgeois:Compactness} by F.~Bourgeois, Y.~Eliashberg, H.~Hofer, C.~Wysocki, and E.~Zehnder and independently~\cite{Cieliebak:Compactness} by K.~Cieliebak and K.~Mohnke. Roughly speaking, neck-stretching is a conformal limit in which the symplectic manifold splits into several pieces, along with the pseudoholomorphic curves that it contains. We are interested in considering the neck-stretching limits of the foliation of pseudoholomorphic lines of $\CP^2,$ which persists for arbitrary compatible almost complex structures by Gromov's classical result~\cite{Gromov:Pseudo}. We work in the same setting of~\cite[Section 3]{Dimitroglou:Isotopy} and direct the reader to that article for the technical details.

A Weinstein neighbourhood of $L$ becomes a concave cylindrical end
$$((-\infty,\log {(4\delta)}] \times UT^*L,d(e^t(\alpha))) \hookrightarrow \left(B^4 \setminus \varphi(B^4(e^\epsilon\sqrt{2/3})),\omega_0\right)$$
after removing the Lagrangian torus, i.e.~when considered in the symplectic manifold $B^4 \setminus L.$ Here $\alpha=pdq|_{T(UT^*L)}$ is chosen to be the contact form on the unit cotangent bundle for a choice of \emph{flat} metric on $\T^2$, and $\delta>0$ is taken sufficiently small. The neck stretching is performed along the hypersurface 
$$\{\log{(3 \delta)}\} \times U T^*L \hookrightarrow B^4 \setminus \varphi(B^4(e^\epsilon\sqrt{2/3}))$$
of contact type that can be seen as a spherical normal bundle of $L$. Stretching the neck amounts to choosing a particular sequence $J_\tau,$ $\tau \ge 0,$ of compatible almost complex structures on $\CP^2$ determined as follows.
\begin{itemize}
\item All $J_\tau$ are fixed outside of the neighbourhood
$$[\log{(2\delta)},\log{(4\delta)}] \times U T^*L$$
of the above spherical normal bundle of $L$;
\item All $J_\tau$ are equal to the standard integrable complex structure $i$ near the divisor $\ell_\infty$, while inside the subset
  $$((-\infty,\log {(2\delta)}] \times UT^*L,d(e^t(\alpha))) $$
  of the concave end they are equal to the almost complex structure $J_{\OP{std}}$ defined in \cite[Section 4]{Dimitroglou:Isotopy}; and
\item Near the spherical normal bundle of $L$ the almost complex structure $J_\tau$ is the pull-back of the cylindrical almost complex structure $J_{\OP{cyl}}$ determined by the flat metric under a (non-symplectic!) diffeomorphism
  $$[\log{(2\delta)},\log{(4\delta)}] \xrightarrow{\cong} [\log{(2\delta)},\log{(4\delta)}+\tau] $$
extended to the $UT^*L$-factor of the symplectisation by the identity.
\end{itemize}
In particular, one obtains a limit compatible almost complex structure on $\CP^2 \setminus L$ that we denote by $J_\infty$ and which is cylindrical on the entire concave end near $L$. We refer to \cite[Sections 3 and 4]{Dimitroglou:Isotopy} for more details.

Now comes the point when we will use the existence of the embedding of the symplectic ball as stipulated by the assumptions of the theorem; we choose the neck-stretching sequence so that
\begin{enumerate} [label={($\dagger$)}, ref={($\dagger$)}]
\item the almost complex structures $J_\tau$ all coincide with the push-forward of the standard almost complex structure $i$ under the symplectomorphism $\varphi$ in the subset $\varphi(B^4(e^{\epsilon}\sqrt{2/3})) \subset B^4 \setminus L$ \label{acs}
\end{enumerate}
holds in addition to the bullet points listed above.

For the analysis that we conduct it is crucial that the cylindrical almost complex structure is chosen with respect to the contact form on $UT^*L$ induced by the \emph{flat} metric on $L$. The reason is that, for instance, the non-existence of contractible geodesics makes the breaking analysis of pseudoholomorphic curves significantly simpler. Recall that the SFT compactness theorem implies that a sequence of finite energy $J_\tau$-holomorphic curves has a subsequence that converges to a pseudoholomorphic building which consists of several levels of punctured finite-energy pseudoholomorphic curves \cite{Bourgeois:Compactness}, \cite{Cieliebak:Compactness}. These finite energy curves are asymptotic to Reeb chords on $UT^*L,$ i.e.~lifted geodesics for the flat metric in the case under consideration.

We will only be interested in the case of a sequence of $J_\tau$-holomorphic degree one curves in $\CP^2$, which are usually called {\bf (pseudoholomorphic) lines}. Recall that for any given tame almost complex structure there exists a unique pseudoholomorphic line through any two given points, or through one fixed point with a given complex tangency, by Gromov's classical result \cite{Gromov:Pseudo}. In addition, any pseudoholomorphic line is embedded. In the case of an SFT-limit of lines the corresponding building consists of:
\begin{itemize}
\item a non-empty {\bf top level} consisting of punctured $J_\infty$-holomorphic spheres in $\CP^2 \setminus L$;
\item a (possibly zero) number of {\bf middle levels} consisting of punctured pseudoholomorphic spheres in $\R \times UT^*L$ for the cylindrical almost complex structure $J_{\OP{cyl}};$ and
\item a (possibly empty) {\bf bottom level} consisting of punctured pseudoholomorphic spheres in $T^*L$ for the almost complex structure $J_{\OP{std}}$ defined in \cite[Section 4]{Dimitroglou:Isotopy},
\end{itemize}
where the punctured spheres moreover can be glued along the punctures to yield a continuous map of degree one from a single sphere into $\CP^2$. See Figure \ref{fig:building} for examples. Of course, it is also possible that the limit curve consists of a single component contained entirely in the top level; this must then be a sphere without any punctures, and we call it {\bf unbroken} (it is a compact $J_\infty$-holomorphic sphere of degree one in the usual sense). By positivity of intersection, established in \cite{McDuff:LocalBehaviour} by D.~McDuff, one can deduce that any component arising in the limit is a (trivial or nontrivial) branched cover of an embedded punctured sphere. Note that the almost complex structures $J_{\OP{std}}$ and $J_{\OP{cyl}}$ used here have the feature that the canonical $\T^2$-action by isometires on the flat torus $L$ lift to an action by biholomorphisms; see \cite[Section 4]{Dimitroglou:Isotopy}.

In the following we will use the (Fredholm) index $\OP{ind}(u)$ of a punctured sphere $u$ to denote the expected dimension of the component of the moduli space which contains it, where the moduli space consists of curves up to reparametrisations considered without asymptotic constraints at the Bott manifolds of Reeb orbits. For a $k$-punctured sphere $u$ inside $\CP^2 \setminus L$ this index can be expressed by
$$\OP{ind}(u)=k-2+\mu(\overline{u})$$
where $\overline{u}$ is its compactification of $u$ to a chain in $\CP^2$ with boundary on $L$ and where $\mu$ denotes the Maslov class. See \cite[Section 3]{Dimitroglou:Isotopy} for more details.

Here we recall some crucial results from \cite{Dimitroglou:Isotopy}.
\begin{lem}[Section 3 \cite{Dimitroglou:Isotopy}]
\label{lem:index}
\begin{enumerate}
\item For any punctured sphere $u$ in $\CP^2\setminus L$, the Fredholm index of a branched cover $\tilde{u}$ satisfies
$$ \OP{ind}(\tilde{u}) \ge d\cdot \OP{ind}(u) $$
where $d \ge 1$ is the degree of the cover;
\item For a generic almost complex structure $J_\infty$ on $\CP^2\setminus L$ it is the case that $\OP{ind}(u) \ge 0$ for any punctured sphere in $\CP^2 \setminus L$, where this index moreover coincides with the dimension of the moduli space which contains $u$ whenever the curve is simply covered. In addition $\OP{ind}(u) \ge 1$ is odd if $u$ is a plane; and
\item If $v_1,\ldots,v_{N_1} \subset \CP^2 \setminus L$ and $w_1,\ldots,w_{N_2}$ are punctured spheres which constitute the components of a broken line, where $w_i$ reside either in middle levels $\R \times UT^*\T^2$ or the bottom level $T^*\T^2$, then the equality
$$ \sum_{i=1}^{N_1} \OP{ind}(v_i)-\sum_{i=1}^{N_2}\chi(w_i)=4 $$
is satisfied, where $\chi(w_i) \le 0$.
\end{enumerate}
\end{lem}
\begin{proof}
The inequality in (1) was shown in the proof of \cite[Lemma 3.3]{Dimitroglou:Isotopy}. Part (2) is simply \cite[Lemma 3.3]{Dimitroglou:Isotopy}. The equality in (3) was shown in the end of the proof of \cite[Proposition 3.5]{Dimitroglou:Isotopy} (the right-hand side here is 4 instead of 2, as in the reference, since the latter considered indices after a generic point constraint).
\end{proof}
Straight-forward topological considerations of the possibilities of buildings in the class of a line in conjunction with the previous lemma now give us the following useful result:
\begin{cor}
\label{cor:index}
Assume that $J_\infty$ is a generic almost complex structure on $\CP^2 \setminus L$. Let $v_1,\ldots,v_{N_1} \subset \CP^2 \setminus L$ and $w_1,\ldots,w_{N_2}$ be punctured spheres which constitute the components of a broken line, where $w_i$ reside either in middle levels $\R \times UT^*\T^2$ or the bottom level $T^*\T^2$. Then
\begin{itemize}
\item There are at least two planes among the top-level components $\{v_i\}$, and any component $v_i$ in the top level satisfies $\OP{ind}(v_i) \in \{0,1,2,3\}$; and
\item If some $v_{i_0}$ satisfies $\OP{ind}(v_{i_0}) \ge 2$, then
\begin{itemize}
\item there are precisely two planes among the components $\{v_i\}$ in the top level,
\item the remaining components $v_i$, $i \neq i_0$, in the top level satisfy $\OP{ind}(v_i) \le 1$, and
\item all components $\{w_i\}$ in the middle and bottom levels are cylinders, while the top components $\{v_i\}$ are either cylinders or planes.
\end{itemize}
\end{itemize}
\end{cor}

\subsection{Extracting an SFT-limit of lines}
Choose a generic point
$$\pt \in \varphi(B^4((e^{\epsilon}-1)\sqrt{2/3}))$$
and consider the $J_\tau$-holomorphic lines that pass through $\pt$ as well as some second fixed point on $L$. (By Gromov's result \cite{Gromov:Pseudo} there is always a unique such line.) A sequence of such lines for which $\tau \to +\infty$ has a convergent subsequence by the SFT compactness theorem~\cite{Bourgeois:Compactness}. Due to the point constraint on $L$ the limit is a pseudoholomorphic building in the class of a ``broken'' line that passes through both $\pt \in B^4 \setminus L$ as well as some point on the torus $L.$


The monotonicity property for the symplectic area of pseudoholomorphic curves (see~\cite{Sikorav:SomeProperties} by J.-C.~Sikorav) applied to the ball
$$\varphi(B^4_\pt(\sqrt{2/3})) \subset \varphi(B^4(e^\epsilon\sqrt{2/3}))$$
while using the property \ref{acs} satisfied by the almost complex structure implies that $$\int_{A_\pt}\omega_{\OP{FS}} \ge \pi 2/3$$ for any component $A_\pt\subset \CP^2 \setminus L$ in the top level of the limit building that passes through the point $\pt.$ In particular, since the total area of these components is $\int_{\ell_\infty}\omega_{\OP{FS}}=\pi$, there is a unique such component. From this we are able to conclude that:
\begin{lem}
\label{lem:plane}
For a generic point $\pt \in \varphi(B^4((e^{\epsilon}-1)\sqrt{2/3}))$ and a generic perturbation of the almost complex structure $J_\infty$ in a unit normal bundle of $L \cup \ell_\infty,$ we can assume that the component $A_\pt$ is
\begin{itemize}
\item disjoint from $\ell_\infty,$ 
\item of symplectic area $2 \pi r^2$ (where thus $r<1/\sqrt{2}$) and
\item of index three (i.e.~Maslov index four) and embedded (thus in particular it is not a branched cover).
\end{itemize}
\end{lem}
\begin{proof}
Every broken pseudoholomorphic line must contain a plane that is disjoint from $\ell_\infty$ by the flatness of the metric on $L$ used in the construction of the neck-stretching sequence; see~\cite[Section 3]{Dimitroglou:Isotopy}. Since the punctured spheres inside $\CP^2 \setminus (\ell_\infty \cup L)$ are of symplectic area equal to $k\pi r^2 > k\pi/3$ for some $k =1,2,3,\ldots,$ by our assumptions on $L$, and since $A_\pt$ is of symplectic area at least $\pi 2/3$ by the above monotonicity argument, we conclude that $A_\pt$ is disjoint from $\ell_\infty$. In fact, we also conclude that its symplectic area is precisely equal to $2\pi r^2$ (i.e.~$k=2$), where thus $r<1/\sqrt{2}$ necessarily holds. Furthermore, there can be no other punctured spheres in the top level that are disjoint from $\ell_\infty$ except $A_{\OP{pt}}$, which implies that $A_{\OP{pt}}$ is a plane. (For these arguments we use the property that a sphere of degree one is of total symplectic area equal to $\pi$ and that every component in the top level of the building contributes positively to the symplectic area, while the other curves contribute zero to the symplectic area.)

The Fredholm index of $A_\pt$ is determined as follows. First we observe that $\OP{ind}(A_\pt)$ is odd and at most equal to three for a generic almost complex structure by Lemma \ref{lem:index} and Corollary \ref{cor:index} (to achieve genericity it suffices to perturb near $L$). Finally, since the point $\pt$ was chosen to be generic, we can assume that $A_\pt$ is of index at least two, and moreover not a branched cover of a plane of index one, as follows by a simple dimension count.
\end{proof}
Recall that the plane $A_{\OP{pt}}$ must be embedded by positivity of intersection \cite{McDuff:LocalBehaviour}, since it is not a branched cover.
\begin{lem}
\label{lem:simple}
The $J_\infty$-holomorphic plane $A_{\OP{pt}} \subset \CP^2 \setminus (\ell_\infty \cup L)$ of Fredholm index three (Maslov index four) produced by the above lemma has a simply covered asymptotic Reeb orbit.
\end{lem}
\begin{proof}
Consider a sequence of $J_\tau$-holomorphic lines which satisfy a generic tangency condition at a generic point $\OP{pt}' \in A_{\OP{pt}}$ as $\tau \to +\infty.$ Using the SFT compactness theorem, we can extract a limit holomorphic building from a convergent subsequence.

We first claim that the limit component is smooth at the point where the tangency is taken. Indeed, positivity of intersection implies that in some neighbourhood of the point $\OP{pt}',$ the underlying simply covered curve must be smooth; see \cite{McDuff:LocalBehaviour}. There is still the possibility that the building contains a multiple cover of $A_{\OP{pt}}$ branched at $\OP{pt}'$ (such a curve satisfies any prescribed tangency condition). This is however not possible, since the symplectic area of a line is equal to $\pi$ which is strictly less than $k\int_{A_{\OP{pt}}}\omega_0$ for any $k>1$.

Taking the tangency condition to be generic, a dimension count together with Corollary \ref{cor:index} implies that the limit in fact must be an \emph{unbroken} $J_\infty$-holomorphic line $\ell \subset \CP^2 \setminus L$ (i.e.~a pseudoholomorphic curve without punctures) that satisfies the tangency. (Any top component of a broken line comes in a moduli space of dimension at most three.) Since the connecting homomorphism
$$H_2(B^4,L)\xrightarrow{\delta} H_1(L)$$
is an isomorphism, we see that $A_{\OP{pt}} \bullet \ell$ is divisible by the multiplicity of the orbit. Positivity of intersection \cite{McDuff:LocalBehaviour} allows us to conclude that
$$0<A_{\OP{pt}} \bullet \ell \le [\ell_\infty] \bullet [\ell_\infty]=1,$$
which shows that this multiplicity is precisely equal to one.
\end{proof}

\subsection{A condition for Hamiltonian unknottedness}

\begin{figure}[htp]
\centering
\vspace{3mm}
\hspace{20mm}
\labellist
\pinlabel $\CP^2\setminus L$ at -17 57
\pinlabel $T^*L$ at -10 18
\pinlabel $\pt$ at 70 78
\pinlabel $A_\pt$ at 88 60
\pinlabel $C$ at 88 18
\pinlabel $B$ at 169 60
\pinlabel $D$ at 171 18
\pinlabel $B_\infty$ at 133 60
\pinlabel $A_\infty$ at 48 60
\pinlabel $\ell_\infty$ at 31 72
\pinlabel $\ell_\infty$ at 115 79
\pinlabel $3$ at 70 53
\pinlabel $1$ at 155 53
\pinlabel $1$ at 30 53
\pinlabel $3$ at 115 53
\pinlabel $\color{blue}\gamma$ at 42 12
\pinlabel $\color{blue}\gamma'$ at 126 12
\pinlabel $2$ at 65 12
\pinlabel $2$ at 150 12
\pinlabel $\color{red}\eta$ at 30 25
\pinlabel $\color{red}-\eta$ at 67 26
\pinlabel $\color{red}\eta'$ at 115 25
\pinlabel $\color{red}-\eta'$ at 152 25
\endlabellist
\includegraphics[scale=1.5]{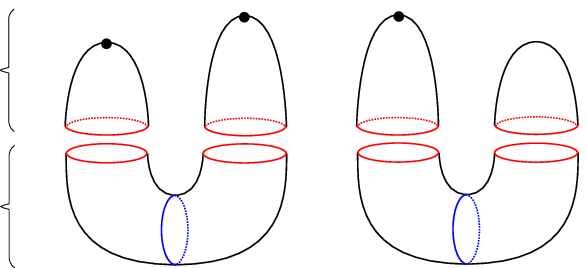}
\hspace{15mm}
\caption{The numbers indicate the Fredholm indices of the components (i.e.~dimension of the moduli space of the respective component without any asymptotic constraint in the Bott manifolds of periodic Reeb orbits). The asymptotic orbits are lifts of the geodesics on $L$ in the homology classes $\pm\eta \in H_1(L)$ and $\pm\eta' \in H_1(L)$. The bottom component are the complexifications of the unoriented closed geodesics $\gamma \subset L$ and $\gamma' \subset L$ in the two different homology classes.}
\label{fig:building}
\end{figure}

The monotonicity combined with Lemma \ref{lem:plane} now implies that the broken line produced in the previous subsection consists of precisely two components in its top level: the embedded plane $A_\pt$ together with an embedded plane $A_\infty$ that passes through $\ell_\infty$; both are simply covered and have simply covered asymptotics. Further, the classification of pseudoholomorphic cylinders in \cite[Section 4]{Dimitroglou:Isotopy} implies that the component in the bottom level is a standard cylinder; roughly speaking, these are the complexifications of the geodesic in the class $\pm \eta \in H_1(L)$ to which the planes are asymptotic. Even if the original broken line does not pass through $L,$ we can replace the cylinder in the bottom level with a cylinder that intersects $L$ cleanly precisely in the corresponding geodesic; such a configuration is shown on the left in Figure \ref{fig:building}.

Since the involved asymptotic orbits are simply covered by Lemma \ref{lem:simple}, the smoothing technique from~\cite[Section 5]{Dimitroglou:Isotopy} can then be used to produce a smoothing of the above building to an embedded symplectic sphere. Moreover, the resulting symplectic sphere intersects $L$ cleanly along the simply covered closed geodesic in class $\pm \eta \in H_1(L)$ to which the planes $A_\pt$ and $A_\infty$ are asymptotic. In other words, the assumptions of the below classification theorem is met, from which the existence of the sought Hamiltonian isotopy then follows.

\begin{thm}
\label{thm:IntersectingLine}
Consider a Lagrangian torus
$$L \subset (B^4,\omega_0)=(\CP^2\setminus \ell_\infty,\omega_{\OP{FS}})$$
for which there exists a tame almost complex structure $J$ on $\CP^2$ that is standard near $\ell_\infty$, and for which some $J$-holomorphic line $\ell$ has the property that $\ell \cap L$ is a simple closed curve $\gamma \subset L$ of Maslov index four (computed using the trivialisation of $TB^4$). Then $L$ is Hamiltonian isotopic inside $(B^4,\omega_0)$ to a product torus.
\end{thm}

\subsection{The proof of Theorem \ref{thm:IntersectingLine}}
By the ``refined'' version of the nearby Lagrangian conjecture for the cotangent bundle $(T^*\T^2,d(p_1d\theta_1+p_2d\theta_2))$ of a torus established in~\cite[Theorem B]{Dimitroglou:Whitney} it suffices to find a Hamiltonian isotopy of $L$ supported inside $B^4$ that places the torus inside the subset
$$ (\CP^2 \setminus (\ell_\infty \cup \{z_1z_2=0\}),\omega_{\OP{FS}}) \cong (\T^2 \times U,d(p_1d\theta_1+p_2d\theta_2))$$
in a position which, moreover, makes it homologically essential inside the same neighbourhood. Here $z_i$ denote the standard affine coordinates on $\CP^2 \setminus \ell_\infty \cong \C^2,$ and
$$U =\{p_1+p_2 < \pi, \:\: p_1,p_2>0\} \subset \R^2$$
is an open convex subset.

To produce a Hamiltonian isotopy of $L$ to the sought position we will rely on the techniques from the proof of \cite[Lemma 5.7]{Dimitroglou:Whitney}, by which it suffices to find two $J$-holomorphic lines $\ell_i \subset \CP^2,$ $i=1,2,$ that intersect $\ell_\infty$ in two distinct points, and for which $L \subset \CP^2 \setminus (\ell_\infty \cup \ell_1 \cup \ell_2)$ is homologically essential. Namely, after a deformation near the nodes of $\ell_\infty \cup \ell_1 \cup \ell_2$ that can be performed by hand, one can show that there exists a Hamiltonian isotopy that fixes $\ell_\infty$ setwise while it takes the union $\ell_\infty \cup \ell_1 \cup \ell_2$ of $J$-holomorphic lines to the three standard lines $\ell_\infty \cup \{z_1z_2=0\}$.

In order to construct the $J$-holomorphic lines $\ell_i$, $i=1,2$, we need to again consider a neck-stretching sequence $J_\tau$ induced by a flat metric on $L.$ It will furthermore be crucial that:
\begin{itemize}
\item $J_\tau=i$ near $\ell_\infty,$ and
\item the line $\ell$, whose existence we are assuming, remains $J_\tau$-holomorphic for all $\tau \ge 0$.
\end{itemize}
In other words, we want $\ell$ to converge to a building as shown on the left in Figure \ref{fig:building} when taking the limit $\tau \to +\infty.$ We use $\pm\eta \in H_1(L)$ to denote the homology class of the unoriented simple closed curve $\gamma =L \cap \ell$ on $L$. The two planes in the top level of the limit building will be asymptotic to the two lifts of a geodesic which coincides with this closed curve for a suitable choice of flat metric on the torus.

To ensure that $J_\tau$ can be made to satisfy the second bullet point above, we need the following intermediate result.
\begin{lem}
\label{lem:intersection}
After a Hamiltonian isotopy supported in a small Weinstein neighbourhood of $L$, the line $\ell$ can be made to coincide with a ``complexified geodesic'' (i.e.~a $J_{\OP{std}}$-holomorphic cylinder explicitly described in \cite[Section 4]{Dimitroglou:Isotopy}) for the flat metric on $L$ inside some even smaller Weinstein neighbourhood.
\end{lem}
\begin{proof}
Recall the standard fact that any smooth isotopy of $L$ can be extended to an ambient Hamiltonian isotopy of its Weinstein neighbourhood $D_{\le \delta}T^*L.$ In this manner we can deform $\ell$ in order to make it intersect $L$ in a closed geodesic in the class $[\ell \cap L] \in H_1(L)$ for any choice of flat metric.

The subspace of linear symplectic two-planes inside $(\C^2,\omega_0)$ that intersect some fixed symplectic two-plane non-transversely is a contractible space. This fact can be readily used to deform $\ell$ in order to, first, make it tangent to the complexification of $\ell \cap L$ and, second, to make it even coincide with the complexification in some tiny neighbourhood.
\end{proof}
By the above we can assume that $\ell$ remains $J_\tau$-holomorphic for all $\tau \ge 0$ as sought, and they thus converge in the SFT-sense to a broken line as shown on the left in Figure \ref{fig:building} as $\tau\to+\infty$. We then combine this result with the existence of broken lines from \cite{Dimitroglou:Isotopy}.
\begin{lem}
\label{lem:brokenlines}
Under the assumption of the existence of the line $\ell$ as above, there is a stretched almost complex structure $J_\infty$ on $\CP^2 \setminus L$ which admits the following broken lines.

First, there is a one-parameter family of buildings parametrised by $t \in (-\epsilon,\epsilon)$ which consist of:
\begin{itemize}
\item {\bf Top level:} Two $J_\infty$-holomorphic planes $A^t_{\OP{pt}},A^t_\infty \subset \CP^2 \setminus L$, where $A^t_{\OP{pt}}$ is of index index 3, passes through $\pt \in B^4$, and is disjoint from $\ell_\infty$, while $A^t_\infty$ is of index 1 and intersects $\ell_\infty$ transversely in a single point; and
\item {\bf Bottom level:} A single $J_{\OP{std}}$-holomorphic cylinder $C^t$ which is the complexification of the geodesic $\gamma_t$.
\end{itemize}
In addition, the closed geodesics $\gamma_t$ are simply covered and provide a foliation the open annulus $\bigcup \gamma_t \subset L$.

Second, there is a building which consists of:
\begin{itemize}
\item {\bf Top level:} Two $J_\infty$-holomorphic planes $B,B_\infty \subset \CP^2 \setminus L$, where $B$ is of index index 1 and is disjoint from $\ell_\infty$, while $B_\infty$ is of index 3 and intersects $\ell_\infty$ transversely in a single point; and
\item {\bf Bottom level:} A single $J_{\OP{std}}$-holomorphic cylinder $D$ which is the complexification of the geodesic $\gamma'$.
\end{itemize}
In addition, the closed geodesic $\gamma'$ intersects each geodesic $\gamma_t$ transversely in a single point.

These two buildings are depicted in Figure \ref{fig:building}.
\end{lem}
\begin{proof}
The building of the first kind for $t=0$ can be taken to be the SFT-limit of the $J_\tau$-holomorphic lines $\ell$ for $\tau \to +\infty$; here we rely on the previous Lemma \ref{lem:intersection}. The existence of the remaining buildings of the same type can be deduced by applying the automatic transversality theorem, which is due to C.~Wendl in the SFT setting \cite[Theorem 1]{Wendl:Automatic}, to the two $J_\infty$-holomorphic planes in the top level of the previously constructed building. To use automatic transversality it crucial that the two components in the top level are immersed planes of positive index. The family of cylinders in the bottom level of the building can then be constructed explicitly; see \cite[Section 4]{Dimitroglou:Isotopy}.

The existence of the building in the second bullet point was established in \cite[Proposition 5.7]{Dimitroglou:Isotopy} by considering a limit of lines with two point constraints: one point on $L$ and one point near $\ell_\infty$.

The intersection property of the geodesics $\gamma_t$ and $\gamma'$ is clear by positivity of intersection: they must intersect since their homology classes are not collinear, while $[\ell_\infty]\bullet[\ell_\infty]=1$ implies that they cannot intersect in more than one point.
\end{proof}

In the remainder of the proof we produce the sought pseudoholomorphic lines by following an idea due to K.~Mohnke \cite{Mohnke:Needle}: extract the sought lines by taking point constraints at suitable pseudoholomorphic planes, in order to achieve the required linking behaviour with $L$. Here it is crucial that both classes of buildings supplied by Lemma \ref{lem:brokenlines} exist; in general we are only guaranteed the second type of broken line described there.

To obtain the lines we will consider the limits of lines that that satisfy a suitable point constraint after stretching the neck. The key technical step is the following lemma.
\begin{lem}
Consider the SFT-limit of the unique $J_\tau$-holomorphic lines that pass through two generic points $\OP{pt}_1 \in A^{t}_{\OP{pt}}$ and $\OP{pt}_2 \in B \cup B_\infty$ as $\tau \to +\infty$. Then this limit is an \emph{unbroken} $J_\infty$-holomorphic line in $\CP^2 \setminus L$ which passes through the same two points.
\end{lem}
\begin{proof}

The limit building is a (possibly unbroken) line which intersects the planes $A^{t'}_{\OP{pt}}$ in a discrete nonempty subset for any $t'$ which is sufficiently close to $t$. Indeed, either the limit intersects $A^{t}_{\OP{pt}}$ in a discrete nonempty subset, in which case the statement follows from positivity of intersection, or the limit building contains a (possibly trivial) multiple cover of the plane $A^{t}_{\OP{pt}}$ itself as a component in its top level. In the second case we need to use the fact from Lemma \ref{lem:brokenlines} that $A^{t}_{\OP{pt}} \cap A^{t'}_{\OP{pt}}=\{\OP{pt}\}$.

If the point $\OP{pt}_1 \in A^{t}_{\OP{pt}}$ is chosen generically, then we can further use Corollary \ref{cor:index} to deduce that the top level component of the limit building that passes through $\OP{pt}_1$ must have an underlying simply covered punctured sphere whose index is at least 2. (A generic such point constraint is of codimension two, and if the component is equal to $A^t_{\OP{pt}}$ itself then the statement is automatically true.)

If the limit is a broken line, then it consists of precisely two planes of positive index together with a number of cylinders in its top, middle, and bottom levels by Corollary \ref{cor:index}. The classification of $J_{\OP{std}}$-holomorphic cylinders from \cite[Section 4]{Dimitroglou:Isotopy} and the computations of their intersection numbers (see Corollary 4.3 therein) implies that middle and bottom levels of the cylinder must have asymptotics in the classes $\pm k\eta \in H_1(L)$ for some $k > 0$. (If not, this broken line would intersect the broken lines of the first kind from Lemma \ref{lem:brokenlines} in more than two points: at least once at $\OP{pt} \in \CP^2 \setminus L$ in the top level, and at least once arising from intersections of cylinders in the middle and bottom levels.)

If the limit line is broken, then by examining the asymptotic orbits of the punctured spheres in the middle and bottom level of the building (these were shown to all be geodesics in the homology classes $\pm k\eta$), we conclude that the building does not contain any of the two planes $B$ and $B_\infty$, nor any branched cover of these. In particular, the limit building intersects the union $B \cup B_\infty$ of planes in a discrete nonempty subset.

However, if the building in fact is broken, then we again arrive at a contradiction by using positivity of intersection. Indeed, the limit building contains a cylinder in its middle and bottom level with asymptotics in homology classes of the form $\pm k\eta$ by the above and, hence, it must intersect the cylinder $D$ in the second building of Lemma \ref{lem:brokenlines} by \cite[Corollary 4.3]{Dimitroglou:Isotopy}. On the other hand, it also intersects $B \cup B_\infty$ in a discrete nonempty subset, so these two buildings intersect with a total algebraic intersection number at least 2, which contradicts $[\ell_\infty]\bullet[\ell_\infty]=1$.
\end{proof}

The above lemma now provides us with the needed pseudoholomorphic lines $\ell_1$ and $\ell_2$. More precisely, the line $\ell_1$ can be taken to be an unbroken $J_\infty$-holomorphic line which passes through the two planes $A^{t}_{\OP{pt}}$ and $B$, while $\ell_2$ can be taken to be the line which passes through the two planes $A^{t}_{\OP{pt}}$ and $B_\infty$. The sought linking properties of $\ell_i$ and $L$ now readily follow from positivity of intersection; these unbroken lines cannot pass through the two planes $B$ and $B_\infty$ simultaneously. The obtained configuration of lines is shown schematically in Figure \ref{fig:building2}.

For suitable point constraints, the two lines $\ell_i$ produced above moreover intersect in a point disjoint from $\ell_\infty.$ We have thus managed to produce the lines in the sought position, where the previously established linking properties imply that $L \subset \CP^2 \setminus (\ell_1 \cup \ell_2 \cup \ell_\infty)$ is homologically essential as needed. (The latter complement of three lines is diffeomorphic to $\T^2 \times \R^2$.)
\qed

\begin{figure}[htp]
\centering
\vspace{3mm}
\hspace{20mm}
\labellist
\pinlabel $\CP^2\setminus L$ at -17 57
\pinlabel $T^*L$ at -10 18
\pinlabel $\pt$ at 70 78
\pinlabel $A^t_\pt$ at 88 60
\pinlabel $C$ at 88 18
\pinlabel $D$ at 168 18
\pinlabel $B_\infty$ at 130 65
\pinlabel $A^t_\infty$ at 48 60
\pinlabel $\ell_\infty$ at 31 72
\pinlabel $\ell_1$ at 114 83
\pinlabel $\ell_\infty$ at 114 66
\pinlabel $\ell_2$ at 92 43
\pinlabel $B$ at 153 53
\pinlabel $\color{red}\eta$ at 30 25
\pinlabel $\color{red}-\eta$ at 67 26
\pinlabel $\color{red}\eta'$ at 115 25
\pinlabel $\color{red}-\eta'$ at 152 25
\endlabellist
\includegraphics[scale=1.5]{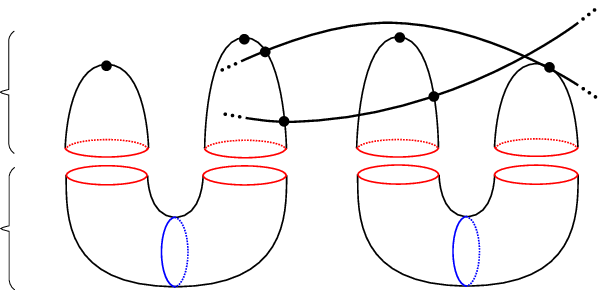}
\hspace{15mm}
\caption{Two unbroken pseudoholomorphic lines $\ell_1$ and $\ell_2$ which pass through the different planes of the two buildings supplied by Lemma \ref{lem:brokenlines}. For the particular configuration shown it follows that $L$ is homologically essential in the complement of $\ell_1 \cup \ell_2 \cup \ell_\infty$.}
\label{fig:building2}
\end{figure}

\section{Proof of Part (2) of Theorem \ref{thm:1}}

It may be possible that one can prove Part (2) by reusing the analysis done in Part (1); however, we here choose to perform a slightly different neck-stretching argument.

First we recall the classical result \cite[Corollary 1.5]{McDuff:FromSymplectic} by D.~McDuff which in particular says that any two symplectic embeddings of closed balls of the same radius into $\CP^2$ are Hamiltonian isotopic. From this it suffices to establish the following proposition, and then use the classification of Lagrangian tori inside a round sphere (see e.g.~\cite[Section 3]{Dimitroglou:Extremal}) which follows from elementary techniques.
\begin{prop}
Any monotone torus $L \subset \CP^2$ which is disjoint from $B^4(\sqrt{2/3}) \subset \CP^2$ is contained entirely inside $S^3(\sqrt{2/3})$.
\end{prop}
\begin{proof}
The proof is based upon the same idea as the proof of the main result in \cite{Dimitroglou:Extremal} by the author, but where the null-homology of $L$ is given by a neck-stretching argument in a slightly different geometric setting. 

Assume the contrary, i.e.~that $L$ is not contained entirely inside $S^3(\sqrt{2/3})$ (in fact, in view of \cite[Theorem 1.6]{Cieliebak:PuncturedHolomorphic}, the torus must at least touch this sphere). As in \cite[Section 4]{Dimitroglou:Extremal} we can then construct a small spherical cotangent fibre
$$S^1 \cong S_0 \subset \CP^2 \setminus (D^4(\sqrt{2/3}) \cup L)$$
contained inside a thin Weinstein neighbourhood of $L$, which thus has the feature that it links $L \subset \CP^2$ nontrivially (recall that $L$ is nullhomologous). From now on the choice of $S_0$ will remain fixed. In addition, for any class of almost complex structures that are fixed in some given neighbourhood of $S_0$, we can now infer that there exists a constant $\hbar>0$ such that:
\begin{enumerate}[label={(M)}, ref={(M)}]
\item Any punctured pseudoholomorphic sphere $A \subset \CP^2 \setminus L$ that passes through $S_0$ satisfies the bound
$$\int_{A \setminus D^4(\sqrt{2/3})}\omega_{\OP{FS}} \ge \hbar>0$$
on its symplectic area; \label{M}
\end{enumerate}
this follows from the monotonicity property \cite{Sikorav:SomeProperties} precisely as \cite[(M)]{Dimitroglou:Extremal}.

We stretch the neck around $L$ while keeping the almost complex structure standard inside $D^4(\sqrt{2/3-\delta})$ as well as in the above neighbourhood of $S_0$. Neck stretching works as described in Section \ref{sec:neckstretch}, with the only difference that we cannot assume that $\ell_\infty$ is holomorphic. The crucial feature of our setup is that we can stretch the neck while keeping the constant $\hbar>0$ above fixed, and simultaneously taking $\delta>0$ to be arbitrarily small.

The same analysis is the proof of Part (1) of Theorem \ref{thm:1} implies that the broken lines which pass through a generic point $\OP{pt} \in D^4(\sqrt{2/3}-\delta)$ close to $0$ must be of area strictly greater than $\pi/3$, and hence consist of a plane $A_{\OP{pt}}$ of index 3 which passes through $\OP{pt}$, an additional plane in the top level of index 1, together with a single cylinder in the bottom component. Here the monotonicity of $L$ significantly simplifies the analysis of the possible breakings. (It does not make sense to make any claims about whether $A_{\OP{pt}}$ is disjoint from $\ell_\infty$ or not, since the latter sphere is not necessarily holomorphic. This is not a problem, since we do not need to show that $A_{\OP{pt}}$ has a simply covered asymptotic.)

The remainder of the proof is the same as in \cite[Section 4.3]{Dimitroglou:Extremal}. We construct a null-homology of $L \subset \CP^2$ by evaluation from the one-dimensional component of the moduli space of planes of index 3 which pass through the point $\OP{pt} \in D^4(\sqrt{2/3-\delta})$ and which contains the plane $A_{\OP{pt}}$. (After a blow-up we can consider the strict transform of the moduli problem, which instead concerns planes of index 1 that intersect the exceptional divisor.) Here it is important that this component of the moduli space is compact, which is a consequence of the monotonicity of $L$ together with the genericity of the choice of point $\OP{pt} \in D^4(\sqrt{2/3}-\delta)$.

Note that the aforementioned planes all have symplectic area $\pi2/3$ by the monotonicity of $L$, while the area of the same planes concentrated inside the ball $D^4(\sqrt{2/3-\delta})$ is equal to at least
$$\int_{A_{\OP{pt}} \cap D^4(\sqrt{2/3-\delta})}\ge \pi(2/3-2\delta)$$
by the monotonicity property for symplectic area \cite{Sikorav:SomeProperties}, while using the assumption that the almost complex structure is standard inside $D^4(\sqrt{2/3-\delta})$.

Since $L$ and $S_0$ are nontrivially linked, some of these planes that constitute the null-homology of $L$ are forced to pass through the subset $S_0$, and in view of \ref{M} their symplectic areas when intersected with $\CP^2 \setminus B^4(\sqrt{2/3}-\delta)$ must hence be bounded from below by $\hbar>0$. Since $\delta>0$ can be taken arbitrarily small in the above construction, we finally arrive at the sought contradiction after choosing $0<\delta<\hbar/(2\pi)$.
\end{proof}

\section{Proof of Theorem \ref{thm:2}}
\label{sec:ProofThm2}

\begin{figure}[htp]
\centering
\vspace{3mm}
\labellist
\pinlabel $1$ at 100 -8
\pinlabel $u_1/\pi$ at 124 3
\pinlabel $u_2/\pi$ at 4 118
\pinlabel $u_2/\pi$ at 160 118
\pinlabel $u_1/\pi$ at 280 3
\pinlabel $1-r$ at 75 -8
\pinlabel $1$ at 255 -8
\pinlabel $1/3$ at 190 -8
\pinlabel $1/3$ at -11 35
\pinlabel $1/2$ at -11 51
\pinlabel $1-r$ at -15 84
\pinlabel $1$ at 148 100
\pinlabel $1/3$ at 142 35
\pinlabel $r$ at 148 19
\pinlabel $1/2$ at 142 51
\pinlabel $1-r$ at 140 84
\pinlabel $1$ at -5 100
\pinlabel $1/3$ at 35 -8
\pinlabel $1-r$ at 230 -8
\pinlabel $P_2$ at 49 20
\pinlabel $\color{red}L_{\OP{Cl}}$ at 30 28
\pinlabel $\color{blue}L$ at 178 52
\endlabellist
\includegraphics[scale=1]{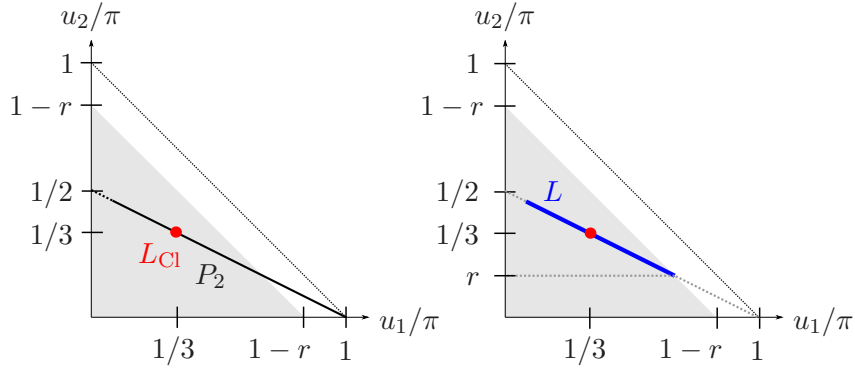}
\vspace{3mm}
\caption{The monotone Clifford torus $L_{\OP{Cl}}=S^1(1/\sqrt{3}) \times S^1(1/\sqrt{3})$ can be isotoped to $L$ inside the probe $P_2 \subset D^4.$ We can moreover place $L$ inside $B^4(\sqrt{1-r})$ for any $r\in(0,1/6).$}
\label{fig:probe1}
\end{figure}

\begin{figure}[htp]
\centering
\vspace{3mm}
\labellist
\pinlabel $\color{red}\gamma_0$ at 152 58
\pinlabel $\color{blue}\gamma_1$ at 163 153
\pinlabel $\sqrt{\frac{1}{3}}$ at 161 104
\pinlabel $\sqrt{\frac{1}{2}}$ at 193 103
\pinlabel $\sqrt{r}$ at 124 100
\pinlabel $x_1$ at 210 91
\pinlabel $iy_1$ at 101 201
\pinlabel $\sqrt{\frac{1}{2}}$ at 115 184
\endlabellist
\includegraphics{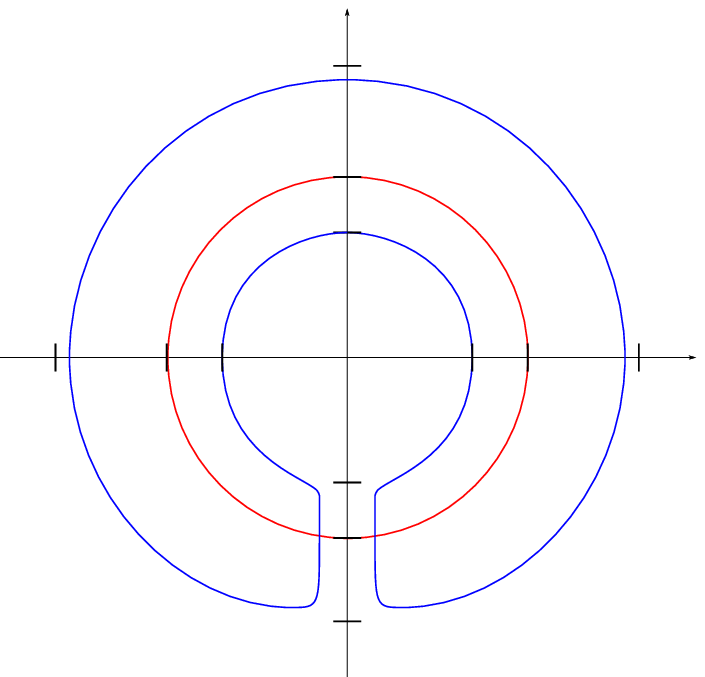}
\caption{Inside the probe $P_2,$ consider a Hamiltonian isotopy of the form $\psi_2(\{\gamma_t\} \times S^1),$ where $\psi_2(\{\gamma_0\} \times S^1)$ is the monotone product torus.}
\label{fig:probe2}
\end{figure}

Probes are a useful tool for constructing Hamiltonian isotopies that was invented by McDuff in~\cite{McDuff:Probes}. For any integer $m=1,2,3,\ldots$ we consider the probe
$$P_{m}:=\mu^{-1}\{u_1=\pi-m u_2,\:\: u_1>0\} \subset \C^2$$
for the standard momentum map
\begin{gather*}
\mu \colon \C^2 \to (\R_{\ge 0})^2,\\
\mu(z_1,z_2):=(\pi\|z_1\|^2,\pi\|z_2\|^2),
\end{gather*}
on $(\C^2,\omega_0).$ We will consider the foliation
\begin{gather*}
\psi_m \colon B^2(1/\sqrt{m}) \times S^1 \xrightarrow{\cong} P_m \subset \C^* \times \C,\\
((r,\theta),\varphi) \mapsto ((\sqrt{1-mr^2},\varphi), (r,\theta+m\varphi)),
\end{gather*}
by symplectic discs $B^2(1/\sqrt{m}) \times \{\varphi\}$ in local angular coordinates. Note that this map indeed extends smoothly over $\{ 0 \} \times S^1.$ Since the symplectic form $\omega_0$ is pulled back to the standard symplectic form $\omega_0$ on $B^2(1/\sqrt{m})$ under the map $\psi_m,$ the characteristic distribution on $P_m$ can be seen to be given by
$$\ker (\omega_0|_{TP_m})=\R\partial_\varphi \subset TP_m.$$
In particular, integrating it, we obtain a trivial symplectic monodromy map on the symplectic disc leaves. This means that
\begin{lem}
For any simple closed curve $\gamma \subset B^2(1/\sqrt{m}),$ the image $\psi_m(\gamma \times S^1) \subset (\C^2,\omega_0)$ is an embedded Lagrangian torus.
\end{lem}
For example, the monotone Clifford torus of symplectic action $\pi/3$ is given by
$$ L_0 := \psi_2(S^1(1/\sqrt{3}) \times S^1) \subset S^3(\sqrt{2/3}).$$
Considering a suitable smooth family of simple closed curves that all bound the area $\pi/3$ inside $B^2(1/\sqrt{3})$ we obtain a Hamiltonian isotopy
$$ L_t := \psi_2(\gamma_t \times S^1) \subset P_2 \subset (\C^2,\omega_0)$$
of Lagrangian tori. We will take $L_0$ to be the Clifford torus while $L_1$ the torus obtained from the curve $\gamma_1 \subset B^2(1/\sqrt{2}) \setminus \{0\}$ shown in Figures \ref{fig:probe1} and \ref{fig:probe2}. Note that this isotopy $L_t$ of tori intersects the subset $S^1 \times \{0\} \subset D^4$ for some $t \in (0,1)$, since not all curves $\gamma_t$ can be disjoint from the origin in $B^2(1/\sqrt{2})$ for obvious topological reasons 

\begin{lem}[Gadbled~\cite{Gadbled:OnExotic}]
The torus $L_1 \subset B^4$ is Hamiltonian isotopic to the Chekanov torus when considered inside the completion
$$(\CP^2,\omega_{\OP{FS}})\supset (\CP^2 \setminus \ell_\infty,\omega_{\OP{FS}})\cong (B^4,\omega_0) \supset L_1$$
of the ball.
\end{lem}
\begin{proof}
The representative of the Chekanov torus described in~\cite{Gadbled:OnExotic} differs from $L_1$ simple by the linear change of coordinates
$$[Z_1:Z_2:Z_3] \mapsto [Z_3,Z_2,Z_1].$$
In particular, the torus is clearly Hamiltonian isotopic to the standard presentation of the Chekanov torus.
\end{proof}

The claim that $L_1$ is not Hamiltonian isotopic to the standard product torus inside $B^4$ then finally follows from the fact that the monotone Clifford and Chekanov tori inside the compactification $(\CP^2,\omega_{\OP{FS}})$ are not Hamiltonian isotopic as shown by Chekanov and Schlenk \cite{Chekanov:TwistTori}. We can thus take $L \coloneqq L_1$. \qed

\subsection{Proof of Theorem \ref{thm:knotted}}
\label{proof:sympknotted}

  The key point is that the Hamiltonian isotopy from $S^1(1/\sqrt{3}) \times S^1(1/\sqrt{3})$ to $L$ inside $\C^2$ that was constructed above can be taken to fix the hypersurface
  $$\{ \|z_1\|^2+2\|z_2\|^2 = 1\} \cap \{z_1 \neq 0\}\subset D^4$$
  setwise; this is the hypersurface that contains the ``probe'' $P_2$ as well as the two Lagrangian tori. To see this, it is convenient to extend the embedding $\psi_2$ of the probe constructed in the same section to a symplectic embedding
\begin{gather*}
\Psi_2 \colon B^2(\sqrt{(1-\delta^2)/2}) \times (-\delta,\delta) \hookrightarrow S^1 \to \C^* \times \C,\\
((r,\theta),(s,\varphi)) \mapsto (s+\sqrt{1-2r^2},\varphi), (r,\theta+2\varphi)),
\end{gather*}
defined using polar coordinate for some small $\delta>0.$ Note that $\omega_0$ pulls back to the product symplectic form $\omega_0 + sds\wedge d\varphi$ on
$$B^2(\sqrt{(1-\delta^2)/2}) \times ((-\delta,\delta) \times S^1),$$
while the restriction $\Psi_2|_{\{s=0\}}=\psi_2$ is the original embedding of the probe $P_2$ constructed above.

One can then realise the Hamiltonian isotopy of the torus in the probe by a suitable lift of a Hamiltonian isotopy of $(B^2(\sqrt{(1-\delta^2)/2}),\omega_0)$ that is generated by a compactly supported Hamiltonian, to yield a Hamiltonian isotopy of the product
$$(B^2(\sqrt{(1-\delta^2)/2}) \times ((-\delta,\delta) \times S^1),\omega_0 + sds\wedge d\varphi )$$
of symplectic manifolds. The sought Hamiltonian is finally produced by multiplication with a suitable smooth bump function.

\qed

\bibliographystyle{plain}
\bibliography{references}

\end{document}